\documentclass[a4paper,12pt]{article}
\title{Characterization of \\
Higher-Order Sobolev Spaces on the Sphere \\
via Generalized Averaging of Function}
\date{\today}
\author{Ikhsan Maulidi\thanks{Department of Mathematics, Universitas Syiah Kuala, Banda Aceh, Indonesia} \ \thanks{Graduate School of Natural Science and Technology, Kanazawa University, Kanazawa, Japan. ikhsanmaulidi@stu.kanazawa-u.ac.jp} and Hiroshi Ohtsuka\thanks{Faculty of Mathematics and Physics, Kanazawa University, Kanazawa, Japan. ohtsuka@se.kanazawa-u.ac.jp}}

\usepackage{amssymb,amsmath,amsthm}
\usepackage{mathrsfs}
\usepackage{esint}
\usepackage{xcolor}
\usepackage{layout}
\usepackage{lineno}
\usepackage{times}
\usepackage{enumerate}
\usepackage{mathtools}
\usepackage{soul} 
\mathtoolsset{showonlyrefs=true}
\numberwithin{equation}{section}

\newcommand{\vep}{\varepsilon}

\newcommand{\lap}{\Delta}

\newcommand{\tend}{\longrightarrow}
\newcommand{\R}{\mathbf{R}}
\newcommand{\C}{\mathbf{C}}
\newcommand{\N}{\mathbf{N}}

\newcommand{\calD}{\mathscr{D}}

\newcommand{\Hn}[2]{H^{#1}\left(#2\right)}

\newcommand{\Le}[2]{L^{#1}\left(#2\right)}

\newcommand{\clo}[1]{\overline{#1}}

\newcommand{\sphere}[1]{\mathbf{S}^{#1}}

\newcommand{\aka}[1]{{\color{black} #1}}

\newcommand{\innerproduct}[2]{\langle #1, #2 \rangle}

\begin{document}
\newtheorem{thm}{Theorem}[section]
\newtheorem{lem}[thm]{Lemma}
\newtheorem{prop}[thm]{Proposition}
\newtheorem{cor}[thm]{Corollary}
\newtheorem{remark}[thm]{Remark}
\newtheorem{defi}[thm]{Definition}

\maketitle
\begin{abstract}
We present a new characterization of higher-order Sobolev spaces on the sphere. Building on the approach of Barcel\'o et al. (2020), we refine the square function they introduced for this purpose. In particular, we provide a detailed analysis of the weight function and the averaging range in its definition. Our results demonstrate that averaging can be performed using a weight function from a broad class, and that the averaging domain can be restricted to a local coordinate patch on the sphere.

\mbox{} 
\end{abstract}

\section{Introduction}\label{sec-intro}

Sobolev spaces play a crucial role in modern mathematical analysis, particularly in the study of partial differential equations (PDEs). These spaces attract considerable attention from scientists due to their ability to integrate the concepts of regularity and integrability. One of the traditional methods commonly used in the analysis of Sobolev spaces is the distributional derivative, which allows for the generalization of classical derivatives to functions that are not smooth. However, this approach has limitations, particularly when applied to non-Euclidean spaces such as metric spaces, where a canonical linear structure is absent, making classical derivatives difficult to define.

As an alternative, Sobolev spaces can be characterized without relying on classical derivatives. One notable approach, introduced by  Haj\l asz (1996) \cite{hl16}, involves pointwise inequalities. This pioneering work inspired many subsequent studies, including the work of Alabern et al. (2012) \cite{amv12}, who developed a characterization of Sobolev spaces using the square function based on ball averages in $\R^d$ with Lebesgue measure. This approach provides an alternative way to define Sobolev norms without relying on classical derivatives.

Building on this idea, Barcel\'o et al. (2020) \cite{blp20} developed a similar characterization for the Sobolev spaces $\Hn{\alpha}{\sphere{d-1}}$ on the $d-1$ dimensional sphere $\sphere{d-1}$ ($d\geq 2$) introducing the following square function:
\begin{equation}
\label{sfs}
S_\alpha(f)^2(\xi):= \int_0^\pi \left|\frac{A_tf(\xi)-f(\xi)}{t^\alpha}\right|^2\frac{dt}{t},\quad \xi\in\sphere{d-1},
\end{equation}
where $f$ is an integrable function on $\sphere{d-1}$ and $A_t f(\xi)$ denotes the average of $f$ over the spherical cap $C(\xi,t)$ centered at $\xi$ with radius $t \in (0,\pi)$:
\begin{align}
A_tf(\xi):=\frac{1}{|C(\xi,t)|}\int_{C(\xi,t)}f(\tau)d\sigma(\tau),\; C(\xi,t):= \{ \eta \in \sphere{d-1} : \xi \cdot \eta \geq \cos t \},\label{avs}
\end{align}
with $|C(\xi,t)|$ is measured by the uniform surface measure $\sigma$ on $\sphere{d-1}$ induced from the Lebesgue measure on $\R^d$.
\begin{thm}[{\cite[Theorem 1.1]{blp20}}]\label{blpth}
For $0<\alpha<2$, the following are equivalent:
\begin{quote}
\rm
(1) $f\in \Hn{\alpha}{\sphere{d-1}}$. \qquad (2) $f\in\Le{2}{\sphere{d-1}}$ and $S_\alpha(f)\in\Le{2}{\sphere{d-1}}$.
\end{quote}
\end{thm}

To address higher-order cases $\alpha \geq 2$, they introduced an extended square function. For $\alpha \in (2n,2(n+1))$ with $n\in\N$ and for integrable functions $g_1,\cdots,g_n$
\begin{equation}\label{eq03b}
S_\alpha(f,g_1,...,g_n)(\xi)^2:=\int_0^\pi \Bigg|A_tf(\xi)-f(\xi)-\sum_{k=1}^ng_k(\xi)A_t(|\xi-\cdot|^{2k})(\xi)\Bigg|^2\frac{dt}{t^{2\alpha+1}},
\end{equation}
and for $\alpha=2n$,
\begin{align*}
S_{2n}(f, g_1, \dots, g_n)(\xi)^2 &:= \int_0^\pi \Bigg| A_tf(\xi)-f(\xi) - \sum_{k=1}^{n-1} g_k(\xi) A_t \left( \left| \xi -\cdot \right|^{2k}\right)(\xi) \nonumber\\ 
&- A_t(g_n)(\xi) A_t \left( \left| \xi - \cdot\right|^{2n} \right)(\xi) \Bigg|^2 \frac{dt}{t^{2\alpha + 1}},
\end{align*}
for $\xi \in \sphere{d-1}.$
\begin{thm}[{\cite[Theorem 1.3]{blp20}}]\label{blpth2}
For $\alpha>2$ and $n\in\N$ satisfying $2n\leq\alpha<2(n+1)$, the following are equivalent:
\begin{quote}
\rm
(1) $f\in \Hn{\alpha}{\sphere{d-1}}$. \\
(2) $f\in\Le{2}{\sphere{d-1}}$ and there exist \(g_1,g_2,...,g_n \in \Le{2}{\sphere{d-1}}\) such that $S_\alpha^{\rho,T}(f,g_1,g_2,...,g_n)\in\Le{2}{\sphere{d-1}}$.
\end{quote}
\end{thm}
Although these theorems appear structurally similar to the results of Alabern et al., the proof techniques used differ significantly. Barcel\'o et al. utilized the concept of spherical harmonics and Legendre polynomials, which are naturally related to the geometric symmetries of the sphere. On the other hand, Alabern et al. utilized the Fourier transform and relied on the geometric symmetries of the Euclidean space $\R^d$. They also provided a more general setting involving  $W^{\alpha,p}(\R^d)$. 

In this paper, we extend the work of Barcel\'o et al. (2020) \cite{blp20} by generalizing the square function through the use of weight functions in the angular direction $[0,T]$ during the averaging process. Unlike Barcel\'o et al., who fixed $T = \pi$ (see Definition \ref{def-weight} for the class of weight functions and Definition \ref{def-ga} for our generalized square function $A^\rho_t$), we consider integrals localized to $[0, T] \subset [0, \pi]$.  This extension allows for varying resolution scales in the averaging process. We note that similar weight function was introduced in \aka{\cite[Cor 5.2]{s17}} in the Euclidean setting $\R^d$ considered by Alabern et al. However, such localization appears to be difficult to replicate in the Euclidean framework they studied. 

Although our proof shares similarities with Barcel\'o et al., we utilize Legendre polynomials in a different coordinate system, which simplifies computations despite the generalization. The primary motivation of this research is to extend the definition of Sobolev spaces to more general compact manifolds, such as surfaces with non-constant curvature or manifolds with complex topological constraints. We believe that weighted ball averages and localized range of angular variable to $[0, T] \subset [0, \pi]$  offer a promising first step toward adapting these arguments to local coordinates on general manifolds.

To formalize our generalization, for a fixed unit vector \(\xi \in \sphere{d-1}\), we define the Euclidean angle \(\theta \in [0, \pi]\) between \(\tau \in \sphere{d-1}\) and \(\xi\). The uniform surface-area measure \(\sigma\) on \(\sphere{d-1}\) can be expressed as  
\[
d\sigma = d\sigma(\theta, \tau') = \sin^{d-2}\theta \, d\theta \, d\sigma',
\]  
where \(\tau' \in \sphere{d-2}\) and \(\sigma'\) denotes the uniform surface-area measure on \(\sphere{d-2}\), cf \cite[(1.17)]{ah12}. 

We introduce the following class of wight functions for generalized averaging:
\begin{defi}[Weight functions] \label{def-weight}
For $T\in (0,\pi]$ and $0<t_0<T_0\leq T$ satisfying $T_0<\pi$,  we set 
$$
 W(T, t_0, T_0):=\left\{\rho\in\Le{2}{0,T}\mid \rho\geq 0\; \text{a.e.},\; \int_{t_0}^{T_0}\rho(\theta)d\theta>0\right\}
$$
\end{defi}
Alternatively, we may assume only
\begin{equation}
\label{rho-cond}
\text{$\rho\geq 0$ a.e. and $\int_0^T\rho(\theta)d\theta>0$}
\end{equation}
for $\rho\in\Le{2}{0,T}$. Indeed if this holds, there exist $t_0$ and $T_0$ satisfying $\rho\in W(T, t_0, T_0)$ from an elementary argument, cf \cite[Lemma 3.3]{s18}. 

We also note that a broader class of measures can be considered for $\rho$. Specifically, $\rho$ may be taken from the set of Borel probability measures on $[0,T]$ that satisfy
$$
\rho([t_0,T_0])>0
$$
for some $t_0$ and $T_0$ satisfying the assumptions in Definition \ref{def-weight}. Consequently, our result also encompasses the $\sphere{d-1}$ version of \cite[Theorem 1.4]{hl17}, which corresponds to the case where $\rho$ is a Dirac measure. However, in this case, the definition of $A^\rho_t$ below becomes slightly more involved, requiring, for example, a regular approximation of $f$ and the pullback of $\rho$ via the dilation map $\theta \mapsto T\theta/t$. For simplicity, we restrict our attention to $\rho \in \Le{2}{0,T}$.

\begin{defi}[The generalized averaging] \label{def-ga}
For a fixed $\rho\in W(T,t_0,T_0)$ and for $f\in\Le{2}{\sphere{d-1}}$ on $\sphere{d-1}$, we define
\begin{equation}\label{ga}
A^\rho_{t}f(\xi):=z_t^{-1}\int_{C(\xi,t)}f(\theta,\tau')\rho\left(\frac{T}{t}\theta\right) d\sigma(\theta,\tau'),
\end{equation}
the generalized averaging of $f$ around $\xi$ with radius $t\in(0,T]$, where $z_t$ is chosen in order to satisfy
\begin{equation*}
z_t^{-1}\int_{C(\xi,t)} \rho\left(\frac{T}{t}\theta\right) d\sigma(\theta,\tau') = 1,
\end{equation*}
that is,
\begin{align}
z_t&=\left|\sphere{d-2}\right|\frac{t}{T}\int_0^T \sin^{d-2}\left(\frac{t\theta}{T}\right) \rho(\theta)d\theta.
\label{sigma-t}
\end{align}
\end{defi}
Under the assumptions that \(\rho\in W(T,t_0,T_0)\), the generalized averaging operator \(A_t^\rho\) is well-defined for all \(f \in L^2(\sphere{d-1})\).  Moreover, the averaging operator \(A_tf\) introduced by Barcel\'o et al. \cite{blp20} corresponds to the special case where \(T = \pi\) and \(\rho(\theta) \equiv 1\) for \(\theta \in [0, \pi]\). In this case, we may choose any $t_0$ and $T_0$ satisfying $0<t_0<T_0<\pi$. 

We fix $T$, $t_0$, $T_0$, and $\rho\in  W(T, t_0, T_0)$ throughout this paper. 

\begin{defi}[The generalized square function] \label{SqFunction}
We define the generalized square function $S^{\rho,T}_\alpha$ for $f\in\Le{2}{\sphere{d-1}}$ as follows: 
\begin{itemize}
\item
For $0<\alpha<2$, we define
\begin{equation} \label{lc01}
S_{\alpha}^{\rho,T}(f)^2(\xi):= \int_0^T \Bigg|\frac{A^\rho_tf(\xi)-f(\xi)}{t^\alpha}\Bigg|^2\frac{dt}{t}.
\end{equation}
\item
For  $\alpha \in (2n,2(n+1))$ with $n\in\N$ and $g_1,\cdots,g_n\in\Le{2}{\sphere{d-1}}$, we define
\begin{align}
&S^{\rho,T}_\alpha(f,g_1,...,g_n)(\xi)^2\nonumber\\
&:=\int_0^T \Bigg|A^\rho_tf(\xi)-f(\xi)-\sum_{k=1}^ng_k(\xi)A^\rho_t(|\xi-\cdot|^{2k})(\xi)\Bigg|^2\frac{dt}{t^{2\alpha+1}},\label{eq03b}
\end{align}
\item
For $\alpha=2n$ and $n\in\N$ and $g_1,\cdots,g_n\in\Le{2}{\sphere{d-1}}$, we define
\begin{align}\label{eq3c}
S_{2n}^{\rho,T}&(f, g_1, \dots, g_n)(\xi)^2 \nonumber\\
&:= \int_0^{T} \Bigg| A^\rho_tf(\xi)-f(\xi) - \sum_{k=1}^{n-1} g_k(\xi) A_t^\rho \left( \left| \xi -\cdot \right|^{2k}\right)(\xi) \nonumber\\ 
&\qquad\qquad- A_t^\rho(g_n)(\xi) A_t^\rho \left( \left| \xi - \cdot\right|^{2n} \right)(\xi) \Bigg|^2 \frac{dt}{t^{2\alpha + 1}}.
\end{align}
\end{itemize}
\end{defi}
Corresponding to Theorems \ref{blpth} we have already the following result: 
\begin{thm}[{\cite[Theorem 1.5]{s18}}]\label{themainshortpaper}
For fixed $T\in (0,\pi]$, $\rho\in\Le{\infty}{0,T}$ satisfying \eqref{rho-cond}, and $0<\alpha<2$, the following are equivalent:
\begin{quote}
\rm
(1) $f\in \Hn{\alpha}{\sphere{d-1}}$. \qquad (2) $f\in\Le{2}{\sphere{d-1}}$ and $S_\alpha^{\rho,T}(f)\in\Le{2}{\sphere{d-1}}$.
\end{quote}
\end{thm}
The main theorem of this paper is the following result, which corresponds to Theorems \ref{blpth2} for higher order Sobolev spaces $\alpha \geq 2$.
\begin{thm}[Main Theorem]\label{main}
For $\rho\in  W(T, t_0, T_0)$ and $2n< \alpha<2(n+1)$, $n\in\N$, the following are equivalent:
\begin{quote}
\rm
(1) $f\in \Hn{\alpha}{\sphere{d-1}}$. \\
(2) $f\in\Le{2}{\sphere{d-1}}$ and there exist \(g_1,g_2,...,g_n \in \Le{2}{\sphere{d-1}}\) such that $S_\alpha^{\rho,T}(f,g_1,g_2,...,g_n)\in\Le{2}{\sphere{d-1}}$.
\end{quote}
The conclusion also holds for $\alpha=2n$, $n=1,2,\ldots$ if $\rho$ satisfies \eqref{r-cond-fine}.
\end{thm}
The condition \eqref{r-cond-fine} is somewhat intricate and will be introduced later. However, we believe that it represents one of the novel contributions of this paper. Examples satisfying \eqref{r-cond-fine} will be provided; see Proposition \ref{rho-example}.

The structure of this paper is as follows. In Section 2, we summarize the setting of our study, including several facts about spherical harmonics and the definition of Sobolev spaces based on spherical harmonics. In Section 3, we present the proof of the main theorem (Theorem \ref{main}).

\section{Overview of Concepts and Notation}\label{sec-setup}
In this paper, we generally follow the notation used in \cite{blp20}. All functions take values in the complex numbers $\C$. The symbol $A \sim B$ denotes both $A \lesssim B$ and $A \gtrsim B$, where $A \lesssim B$ means that there exists a constant $C > 0$, independent of the parameters associated with the quantities $A$ and $B$, such that $A \leq CB$. When we wish to specify the dependence of the constant $C$ on certain parameters, we write $\lesssim_{d,\rho,T}$ and $\sim_{d,\rho,T}$, etc. Since we fix $\rho\in W(T,t_0,T_0)$, we abbreviate $\sim_{\rho,t_0,T_0,T}$ as  $\sim_{\rho, T}$ if it makes no confusion

\begin{lem}[{\cite[Lemma 3.1]{blp20}}, see also {\cite[Theorem 2.8]{ah12}}]\label{prop-L}
Let $\xi \in \mathbb{S}^{d-1}$ and $L \in \mathbb{H}_l^d$ be such that $L(R\eta)=L(\eta)$ for every rotation $R$ in $\mathbb{R}^d$ satisfying $R(\xi)=\xi$. Then 
\[
L(\eta)=L(\xi)P_{l,d}(\eta \cdot \xi),
\]
where $P_{l,d}$ denotes the Legendre polynomial of degree $l$ in $d$ dimensions.
\end{lem}

The explicit formula for $P_{l,d}$ is given in \cite[(2.19)]{ah12}:
\begin{equation}
\label{eq19b}
P_{l,d}(s)=l!\Gamma\left(\frac{d-1}{2}\right)\sum_{k=0}^{[l/2]}(-1)^k\frac{(1-s^2)^ks^{l-2k}}{4^kk!(l-2k)!\Gamma(k+\frac{d-1}{2})}.
\end{equation}
It is known that
\begin{equation}\label{P(1)}
P_{l,d}(1)=1\quad\text{for } l=0,1,\dots.
\end{equation}
For later use, we recall the following estimates:
\begin{equation}\label{eq19b2}
P_{l,d}^{(k)}(1)=\frac{l!(l+k+d-3)!\Gamma\left(\frac{d-1}{2}\right)}{2^k(l-k)!(l+d-3)!\Gamma\left( k+\frac{d-1}{2} \right)} 
\sim l^k(l+d-2)^k.
\end{equation}
For every \(t \in [-1,1]\) and $k\in\mathbb{N}\cup\{0\}$, we have
\begin{equation}\label{eq20}
|P_{l,d}^{(k)}(t)|\leq P_{l,d}^{(k)}(1) \sim l^k(l+d-2)^k\sim l^{2k},
\end{equation}
see \cite[(18)(44)]{blp20} and \cite[pp.~58--59]{ah12}. 

We also recall the following estimate from \cite[p.~18]{blp20}, which will be used later:
\begin{equation}\label{Psharp}
\left|P_{l,d}\left(\cos\theta\right)\right|\sin^{d-2}\left(\theta\right) \lesssim_d \theta^{\frac{d-2}{2}} l^{\frac{2-d}{2}} \quad\text{for } \frac{1}{l} \lesssim \theta < \frac{\pi}{4}.
\end{equation}

The addition theorem \cite[Theorem 2.9]{ah12} states:
\[
\sum_{j=1}^{\nu(l)}Y_l^j(\eta)\overline{Y_l^j(\tau)}=\frac{\nu(l)}{|\mathbb{S}^{d-1}|}P_{l,d}(\eta\cdot\tau).
\]
As a consequence, the following reproducing formula holds for $Y_l\in \mathbb{H}_l^d$ \cite[p.~23]{ah12}:
\begin{equation}
\label{repro}
Y_l(\xi)=\frac{\nu(l)}{|\mathbb{S}^{d-1}|}\int_{\mathbb{S}^{d-1}}P_{l,d}(\xi\cdot\eta)Y_l(\eta)d\sigma(\eta).
\end{equation}

Each spherical harmonic $Y_l^j$ is an eigenfunction of the Laplace--Beltrami operator $-\Delta$ satisfying
\begin{equation}
\label{eigen}
-\Delta Y_l^j=l(l+d-2)Y_l^j.
\end{equation}
Therefore, for $f\in C^\infty(\mathbb{S}^{d-1})$ with expansion $f=\sum_{l=0}^\infty\sum_{j=1}^{\nu(l)}\hat{f}_{lj}Y_l^j$, where $\hat{f}_{lj}:=\int_{\mathbb{S}^{d-1}}f\overline{Y_j^l}d\sigma$, we have
\begin{equation}\label{sumeigen}
\|\nabla f\|^2_{L^2(\mathbb{S}^{d-1})}=\sum_{l=0}^\infty l(l+d-2)\sum_{j=1}^{\nu(l)}|\hat{f}_{lj}|^2.
\end{equation}
Moreover,
\[
\sum_{j=1}^{\nu(l)}|\hat{f}_{lj}|^2=\left\|\sum_{j=1}^{\nu(l)}\hat{f}_{lj}Y_l^j\right\|^2_{L^2(\mathbb{S}^{d-1})}=O(l^{d-2k-2})
\]
if $f\in C^k(\mathbb{S}^{d-1})$ \cite[(3.26)]{ah12}. 

Using \eqref{eigen}, we define
\begin{equation}\label{eqn1214}
(-\Delta)^{\frac{\alpha}{2}}f:=\sum_{l=0}^\infty \left\{l(l+d-2)\right\}^{\frac{\alpha}{2}}\sum_{j=1}^{\nu(l)}\hat{f}_{lj}Y_l^j \quad\text{in } L^2(\mathbb{S}^{d-1}),
\end{equation}
and the corresponding norm:
\begin{equation}
\label{sobolev-norm}
\left\|(-\Delta)^{\frac{\alpha}{2}} f\right\|_{L^2(\mathbb{S}^{d-1})}^2=\sum_{l=0}^\infty \left\{l(l+d-2)\right\}^\alpha\sum_{j=1}^{\nu(l)}|\hat{f}_{lj}|^2.
\end{equation}
In particular, $\left\|(-\Delta)^{\frac{1}{2}} f\right\|_{L^2(\mathbb{S}^{d-1})}=\|\nabla f\|_{L^2(\mathbb{S}^{d-1})}$. Following \cite{blp20}, we define $H^\alpha(\mathbb{S}^{d-1})$ as the completion of $C^\infty(\mathbb{S}^{d-1})$ with respect to the norm
\begin{align*}
\|f\|^2_{H^\alpha(\mathbb{S}^{d-1})}&:=\sum_{l=0}^\infty\left(1+l^{1/2}(l+d-2)^{1/2}\right)^{2\alpha}\sum_{j=1}^{\nu(l)}|\hat{f}_{lj}|^2\\
&=\left\|\left(I+(-\Delta)^{1/2}\right)^\alpha f\right\|^2_{L^2(\mathbb{S}^{d-1})},
\end{align*}
see also \cite[Definition 3.23]{ah12}.

\section{Proof of Main Theorem (Theorem \ref{main})}\label{sec-proof}
\subsection{Preliminaries}
The purpose of this section is to prove the following fact:
\begin{prop}\label{equinorm}
Suppsoe that  \(\alpha > 0\) and the non-negative integer  \(n\) satisfy \(2n < \alpha < 2(n+1)\) and \(f \in H^{2n}(\sphere{d-1})\). Then, if \(n = 0\), it holds that
\[
\|(-\lap)^{\frac{\alpha}{2}}f\|_{\Le{2}{\sphere{d-1}}} \sim_{d,\rho, T} \|S_{\alpha}^{\rho,T}(f)\|_{L^{2}(\sphere{d-1})}.
\] 
 If \(n \geq 1\),  it holds that 
\[
\|(-\lap)^{\frac{\alpha}{2}}f\|_{\Le{2}{\sphere{d-1}}} \sim_{d,n, \rho, T} \left\|S_{\alpha}^{\rho,T}\big(f, T_1((-\Delta)f), \ldots, T_n((-\Delta)^n f)\big)\right\|_{L^{2}(\sphere{d-1})}.
\]
The conclusions also holds when $\alpha=2n$ if $\rho$ satisfy \eqref{r-cond-fine}.
\end{prop}

In order to prove this, we prepare the several facts in this subsection. We start from recalling the following fact
\begin{lem}
\label{rho-bound}
For every $\rho\in  W(T, t_0, T_0)$, it holds that
\begin{equation}
\label{z-behavior}
z_t\sim_{d,\rho,T}t^{d-1}
\end{equation}
uniformly for every $t\in[0,T]$. 
\end{lem}
\begin{lem}[cf. {\cite[Lemma 2.1]{blp20}}] \label{coef-A}
For any $Y_l \in \mathbb{H}_l^d$ it holds that 
$$
A^\rho_tY_l=m^\rho_{l,t}Y_l,
$$
where
\begin{equation}\label{eq05}
m^\rho_{l,t}:= \frac{ |\sphere{d-2}|}{z_t}\cdot\frac{t}{T} \int_{0}^{T} P_{l,d}\left(\cos\left(\frac{t}{T}\theta\right)\right)\sin ^{d-2}\left(\frac{t}{T}\theta\right)  \, \rho(\theta)d\theta.
\end{equation}
Moreover it holds that
\begin{equation}\label{m_rough}
|m^\rho_{n,l}|\leq 1.
\end{equation}
\end{lem}
We omit the proofs of these lemmas because they are almost same to that of \cite[Lemma 3.4]{s18} \aka{and} \cite[Lemma 3.1]{s18}.  We note that since $\{m_{l,t}^\rho\}_{l=0}^\infty\subset\R$ forms a bounded sequence, it holds that
$$
A^\rho_t f=\sum_{l=0}^\infty m^\rho_{l,t}\sum_{j=1}^{\nu(l)}\hat{f}_{lj}Y_l^j\quad\text{in $\Le{2}{\sphere{d-1}}$}.
$$
for every $f=\sum_{l=0}^\infty\sum_{j=1}^{\nu(l)}\hat{f}_{lj}Y_l^j\in\Le{2}{\sphere{d-1}}$.

From direct calculations, we get the following formula: 
\begin{lem}
For $k\in\N\cup\{0\}$, it holds that
\begin{align}
&A^\rho_t\left(\lvert\xi-\cdot\rvert^{2k}\right)(\xi)\nonumber\\
&\qquad=2^k\cdot\frac{|\sphere{d-2}|}{z_t}\cdot\frac{t}{T}\int_0^T \left(1-\cos\left(\frac{t}{T}\theta\right)\right)^k\sin^{d-2}\left(\frac{t}{T}\theta\right)\rho(\theta)\, d\theta.\label{Axi}\nonumber
\end{align}
Moreover it holds that
\begin{equation}
\left|A^\rho_t\left(\lvert\xi-\cdot\rvert^{2k}\right)(\xi)\right|
\sim_{d,k,\rho,T} t^{2k}. \label{rough-A}
\end{equation}
\end{lem}
\begin{proof}
\begin{align*}
A^\rho_t&\left(\lvert\xi-\aka{\cdot}\rvert^{2k}\right)(\xi) =z_t^{-1} \int_{C(\xi,t)} \lvert\xi-\tau\rvert^{2k}\rho\left(\frac{T}{t}\theta\right)\,d\sigma(\tau) \\
&=2^k\cdot\frac{|\sphere{d-2}|}{z_t}\int_{0}^{t}(1-\cos\theta)^k\,\sin^{d-2}\theta\rho\left(\frac{T}{t}\theta\right)  \,d\theta\\
&=2^k\cdot\frac{|\sphere{d-2}|}{z_t}\cdot\frac{t}{T}\int_0^T \left(1-\cos\left(\frac{t}{T}\theta\right)\right)^k\sin^{d-2}\left(\frac{t}{T}\theta\right)\rho(\theta)\, d\theta.
\end{align*}
The estimate \eqref{rough-A} is obvious from the definition \eqref{sigma-t} of $z_t$ and the following estimates:
\begin{align*}
A^\rho_t\left(\lvert\xi-\cdot\rvert^{2k}\right)(\xi)&\leq 2^k\cdot\frac{ |\sphere{d-2}|}{z_t}\cdot\frac{t}{T} \int_{0}^{T}\frac{\sin^{2k+d-2}\frac{t}{T}\theta}{\left(1+\cos\left(\frac{t}{T}\theta\right) \right)^{2k}}\rho(\theta) d\theta
\end{align*}
and
\begin{align*}
&A^\rho_t\left(\lvert\xi-\cdot\rvert^{2k}\right)(\xi)
\geq \frac{ |\sphere{d-2}|}{z_t}\cdot\frac{t}{T} \int_{0}^{T}\sin^{2k+d-2}\left(\frac{t}{T}\theta\right)\rho(\theta) d\theta\\
&\geq \frac{ |\sphere{d-2}|}{z_t}\cdot\left(\frac{t}{T} \right)^{2k+d-1}\cdot\left(\frac{\sin T_0}{T_0}\right)^{2k+d-2}\int_{0}^{T}\theta^{2k+d-2}\rho(\theta) d\theta.
\end{align*}
\end{proof}
Here we recall the Taylor formula: 
\begin{equation}\label{PTaylor}
P_{l,d}\left(s\right)= \sum_{k=0}^{n}c_{k,l}\left(1-s\right)^k +(-1)^{n+1}\frac{P_{l,d}^{(n+1)}(\tau(s))}{(n+1)!}\left(1-s\right)^{n+1}
\end{equation}
for some $\tau(s)\in \left(s,1\right)$, where
\begin{equation}\label{eq2025040802}
c_{k,l}=\frac{(-1)^{k}P_{l,d}^{(k)}(1)}{k!}.
\end{equation}
We note the following fact, which is obvious from \eqref{eq19b}:
\begin{prop}\label{c-est}
$|c_{k.l}|\stackrel{<}{\sim}_{d,n,T} l^{2k}$.
\end{prop}

We apply \eqref{PTaylor} with $s=\cos\left(\frac{t}{T}\theta\right)$ and get 
\begin{align}
m_{l,t}^\rho=&\sum_{k=0}^n \frac{c_{k,l}}{2^k}A^\rho_t\left(\lvert\xi-\cdot\rvert^{2k}\right)(\xi)
\\
&+\frac{ |\sphere{d-2}|}{z_t}\cdot\frac{t}{T}\int_0^T(-1)^{n+1}\frac{P_{l,d}^{(n+1)}(\tau(s))}{(n+1)!}\left(1-\cos\left(\frac{t}{T}\theta\right)\right)^{n+1}\\
&\qquad\qquad\times\sin^{d-2}\left(\frac{t}{T}\theta\right)\rho(\theta)d\theta.
\label{mTaylor}
\end{align}

Here we define the operator that connect the Taylor expansion and the spherical harmonics expansion:
\begin{lem}\label{T}
Let
\begin{equation}\label{eq30}
T_k(f) :=\sum_{l=1}^{\infty}\frac{c_{k,l}}{2^k\{l(l+d-2)\}^k}\sum_{j=1}^{v(l)}\hat{f}_{lj}Y_l(\xi), \hspace{1cm} k=1,2,\ldots,n.
\end{equation}
Then \( T_k \) and \( T_k^{-1} \) are continuous in the space \( L^2(\sphere{d-1}) \). 
\end{lem}
\begin{proof}
We may only note that
$$
\frac{c_{k,l}}{2^k\{l(l+d-2)\}^k}\sim_{d, n} 1
$$
for a fixed $k=1,\ldots,n$, which holds from \eqref{eq19b}. 
\end{proof}
We note that
$$
T_k((-\lap)^kf)=\sum_{l=1}^{\infty} \frac{c_{k,l}}{2^k}\sum_{j=1}^{v(l)}\hat{f}_{lj}Y_l(\xi).
$$
Therefore we get
\begin{align}
&A^\rho_tf(\xi)-\sum_{k=0}^{n}T_k\left((-\Delta)^kf\right)(\xi)	A^\rho_t\left(\lvert\xi-\cdot\rvert^{2k}\right)(\xi) \nonumber\\
&=\sum_{l=1}^\infty\left(m_{l,t}^\rho-\sum_{k=0}^{n}\frac{c_{k,l}}{2^k}A^\rho_t\left(\lvert\xi-\cdot\rvert^{2k}\right)(\xi)\right)\sum_{j=1}^{v(l)}\hat{f}_{lj}Y_l^j(\xi).
\label{Aexpansion}
\end{align}

\begin{lem}[cf. {\cite[ p.7 (24)]{blp20}} ]\label{SqRep}
For $f=\sum_{l=0}^\infty\sum_{j=1}^{\nu(l)}\hat{f}_{lj}Y_l^j\in\Le{2}{\sphere{d-1}}$ and $n \in \N \cup \{0\}$, it holds that
\begin{align}
&\left\|S_{\alpha}^{\rho,T}\big(f, T_1((-\Delta)f), \ldots, T_n((-\Delta)^n f)\big)\right\|_{L^{2}(\sphere{d-1})}\nonumber\\
&\qquad=
\begin{cases} 
\displaystyle \sum_{l=1}^\infty I_{\alpha,n}^{\rho,T}(l)\left(\sum_{j=1}^{\nu(l)} |\hat{f}_{lj}|^2\right), &  2n < \alpha < 2(n+1), \\[10pt]
\displaystyle \sum_{l=1}^\infty J_n^{\rho,T}(l)\left(\sum_{j=1}^{\nu(l)} |\hat{f}_{lj}|^2\right), & \alpha=2n, n\neq0,
\end{cases}
\label{lc_piecewise}
\end{align}
where
\begin{align*}
&I_{\alpha,n}^{\rho,T}(l):=\int_0^T | M^\rho_{n,l,t}|^2 \frac{dt}{t^{2\alpha+1}}, \quad J_n^{\rho,T}(l):=\int_0^T | N^\rho_{n,l,t}|^2 \frac{dt}{t^{4n+1}},
\end{align*}
and
\begin{align*}
M^\rho_{n,l,t}:=&
m_{l,t}^\rho-\sum_{k=0}^{n}\frac{c_{k,l}}{2^k}A^\rho_t\left(\lvert\xi-\cdot\rvert^{2k}\right)(\xi),
\\
N^\rho_{n,l,t}:=&M^\rho_{n-1,l,t}-\frac{c_{n,l}}{2^n}A^\rho_t\left(\lvert\xi-\cdot\rvert^{2n}\right)(\xi)\cdot m^\rho_{n,l}\aka{.}
\end{align*}
\end{lem}
\begin{proof}
For case $2n<\alpha<2(n+1), n=0,1,2,\ldots$, we use
\eqref{Aexpansion} and get
\begin{align}
&A^\rho_tf(\xi)-\sum_{k=0}^{n}T_k\left((-\Delta)^kf\right)(\xi)	A^\rho_t\left(\lvert\xi-\cdot\rvert^{2k}\right)(\xi) \nonumber\\
&=\sum_{l=1}^\infty M_{n,l,t}^\rho\sum_{j=1}^{v(l)}\hat{f}_{lj}Y_l^j(\xi) \label{eq31}.
\end{align}
Then the conclusion can be obtained by Fubini's theorem and the orthogonality of $\{Y_l^j\}$ in $\Le{2}{\sphere{d-1}}$.

The cases \(\alpha=2n, n\neq0\) proceeds similarly.
\end{proof}

We note that
$$
M^\rho_{0,l,t}=m^\rho_{l,t}-1.
$$
From the definition of $M^{\rho}_{n,l,t}$ and  \eqref{mTaylor}, we get
\begin{align}
M_{n,l,t}^\rho=&\frac{ |\sphere{d-2}|}{z_t}\cdot\frac{t}{T}\int_{0}^{T} \frac{(-1)^{n+1}P_{l,d}^{(n+1)}(\tilde\tau(\theta))}{(n+1)!}\\
&\qquad\times\left(1-\cos\left(\frac{t}{T}\theta\right)\right)^{n+1}\sin ^{d-2}\left(\frac{t}{T}\theta\right)  \, \rho(\theta)d\theta.\label{MTaylor}
\end{align}
On the other hand, it holds that 
\begin{equation}\label{N-decomp}
N^\rho_{n,l,t}=M^\rho_{n,l,t}-\frac{c_{n,l}}{2^n}A^\rho_t\left(\lvert\xi-\cdot\rvert^{2n}\right)(\xi)M^\rho_{0,l,t}
\end{equation}
and here we will also use the following deformation:
\begin{align}
M^{\rho}_{n,l,t}=&M_{n-1,l,t}^\rho-\frac{c_{n,l}}{2^n}A^\rho_t\left(\lvert\xi-\cdot\rvert^{2n}\right)(\xi)\nonumber\\
=&\frac{ |\sphere{d-2}|}{z_t}\cdot\frac{t}{T}\cdot\frac{(-1)^n}{n!}\int_{0}^{T}\left(P_{l,d}^{(n)}(\tilde\tau(\theta))-P_{l,d}^{(n)}(1)\right)\nonumber\\
&\qquad\times\left(1-\cos\left(\frac{t}{T}\theta\right)\right)^{n}\sin ^{d-2}\left(\frac{t}{T}\theta\right)  \, \rho(\theta)d\theta\nonumber\\
=&\frac{ |\sphere{d-2}|}{z_t}\cdot\frac{t}{T}\cdot\frac{(-1)^n}{n!}\int_{0}^{T}P_{l,d}^{(n+1)}(\tilde\tau_2(\theta))\nonumber\\
&\qquad\times\left(1-\cos\left(\frac{t}{T}\theta\right)\right)^{n+1}\sin ^{d-2}\left(\frac{t}{T}\theta\right)  \, \rho(\theta)d\theta
\label{NTaylor3}
\end{align}
for $\tilde\tau_2(\theta)\in (\tilde\tau(\theta),1)$.

Proposition \ref{equinorm} follows from \eqref{sobolev-norm} and the following facts:
\begin{equation}\label{lc04}
I_{\alpha,n}^{\rho,T} (l):=\int_0^T |M^\rho_{n,l,t}|^2 \frac{dt}{t^{2\alpha+1}} \sim_{d,n,\rho,T} l^{2\alpha} \sim_{d,n,\rho,T} \left\{l(l+d-2)\right\}^\alpha,
\end{equation}
and
\begin{equation}\label{lc04b}
J_n^{\rho,T} (l):=\int_0^T |N^\rho_{n,l,t}|^2 \frac{dt}{t^{4n+1}} \sim_{d,n,\rho,T} l^{4n} \sim_{d,n,\rho,T} \left\{l(l+d-2)\right\}^{2n}.
\end{equation}

The proof of these facts are divided into following Lemma  \ref{LocUpperboundI}, Lemma \ref{LocUpperboundJ}, Lemma \ref{LocLowerboundI}, and Lemma \ref{LocLowerboundJ} which give the upper bound and the lower bound for \(I_{\alpha,n}^{\rho,T}\) and \(J_n^{\rho,T}\) .  We note that the estimate on \(J_n^{\rho,T}\) ($\alpha=2n$) requires sharper argument than that of  \(I_{\alpha,n}^{\rho,T}\)($\alpha\not= 2n$) and require the condition \eqref{r-cond-fine}.

\subsection{Upper bound of \(I_{\alpha,n}^{\rho,T}\)}
\begin{lem}\label{LocUpperboundI}
For $l\in\N$ and $2n < \alpha < 2(n+1), n=0,1,2,...$, it holds that
$$
I_{\alpha,n}^{\rho,T}(l)\stackrel{<}{\sim}_{d,n,\rho,T}l^{2\alpha}.  
$$ 
\end{lem}
\begin{proof}
Take $a\in (0,T)$ and set $t_l:=\frac{a}{l}$  for $l \in \N$. Then $t_l\leq T$ holds for every $l \in \N$. It follows that
\begin{align*}
I_{\alpha,n}^{\rho,T}(l)&=\int_0^T | M^\rho_{l,t}|^2 \frac{dt}{t^{2\alpha+1}} \\
&\leq \int_0^{t_l} | M^\rho_{n,l,t}|^2 \frac{dt}{t^{2\alpha+1}}+ \int_{t_l}^
{T} | M^\rho_{n,l,t}|^2 \frac{dt}{t^{2\alpha+1}}=: 	I_{\alpha,n}^{\rho,T,1}(l)+	I_{\alpha,n}^{\rho,T,2}(l).
\end{align*}
Using \eqref{eq19b2}, \eqref{eq20}, \eqref{z-behavior}, and \eqref{MTaylor}, we obtain
\begin{align}
|M_{n,l,t}^\rho|&\stackrel{<}{\sim}_{d,n,T}\frac{ |\sphere{d-2}|}{z_t}\cdot\frac{t}{T} \cdot l^{2n+2}\int_{0}^{T}\frac{\sin^{2n+d}\frac{t}{T}\theta}{\left(1+\cos\left(\frac{t}{T}\theta\right) \right)^{n+1}}\rho(\theta) d\theta\nonumber\\
&\leq\frac{ |\sphere{d-2}|}{z_t}\cdot\frac{t}{T} \cdot l^{2n+2}\int_{0}^{T}\left(\frac{t}{T}\theta\right)^{2n+d}\rho(\theta) d\theta\\
&\stackrel{<}{\sim}_{d,n,\rho,T} l^{2n+2}t^{2n+2}\quad\text{ for $n=0,1,\ldots$}\label{M-rough-est}
\end{align}
for $0<t<T$. Therefore it holds that
\begin{align}\label{eqn01}
I_{\alpha,n}^{\rho,T,1}(l)
&\stackrel{<}{\sim}_{d,n,\rho,T}l^{4n+4} \int_0^{t_l} t^{3+4n-2\alpha}dt \stackrel{<}{\sim}_{d,n,\rho,T} l^{2\alpha}.
\end{align}
Then, from \eqref{M-rough-est} and \eqref{rough-A}, we have
\begin{align}
|M^\rho_{n, l,t}|&\leq|M^\rho_{n-1,l,t}|+\frac{|c_{n,l}|}{2^n}\left|A^\rho_t\left(\lvert\xi-\cdot\rvert^{2n}\right)(\xi)\right|\\
&\stackrel{<}{\sim}_{d,n,\rho, T}l^{2n}t^{2n}+l^{2n}t^{2n}\sim_{d,n,\rho, T}l^{2n}t^{2n}\label{eqn02}
\end{align}
for $0\leq t\leq T$. Since $2n<\alpha$, we get
\begin{align}\label{eqn03}
I_{\alpha,n}^{\rho,T,2}(l)&=\int_{t_l}^T |M^\rho_{n,l,t}|^2\frac{dt}{t^{2\alpha+1}}\stackrel{<}{\sim}_{d,n,\rho,T}  l^{4n}\int_{t_l}^T t^{4n}\frac{dt}{t^{2\alpha+1}} \stackrel{<}{\sim}_{d,n,\rho,T} l^{2\alpha}.
\end{align}
The merging of \eqref{eqn01} and \eqref{eqn03} produces the proof of this lemma. 
\end{proof}

\subsection{Upper bound of \(J_{n}^{\rho,T}(l)\)}
\begin{lem}\label{LocUpperboundJ}
For $l, n \in\N$, it holds that
$$
J_{n}^{\rho,T}(l)\stackrel{<}{\sim}_{d,n,\rho,T}l^{4n}.  
$$ 
\end{lem}
\begin{proof}
Similar to the proof of Lemma \ref{LocUpperboundI}, we split the integration into two parts:
\begin{equation}\label{eq15a}
J_n^{\rho, T}(l)=\int_0^{t_l} \left|N_{n,l,t}^\rho\right|^2 \frac{dt}{t^{4n+1}}+\int_{t_l}^T \left|N_{n,l,t}^\rho\right|^2 \frac{dt}{t^{4n+1}}=:J^{\rho, T, 1}_n(l)+J^{\rho, T, 2}_n(l),
\end{equation}
with $t_l=\frac{a}{l}$ and $a>0$.

Using \eqref{N-decomp},  \eqref{M-rough-est}, \eqref{rough-A},  and Proposition \ref{c-est}, we obtain
\begin{align*}
|N_{n,l,t}^\rho|&\leq |M^\rho_{n,l,t}|+\frac{|c_{n,l}|}{2^n}\cdot|A^\rho_t\left(\lvert\xi-\cdot\rvert^{2n}\right)(\xi)|\cdot|M^\rho_{0,l,t}|
\nonumber\\
&\stackrel{<}{\sim}_{d,n,T}l^{2n+2}t^{2n+2}+l^{2n}t^{2n}\cdot l^2 t^2\stackrel{<}{\sim}_{d,n,\rho,T} l^{2n+2}t^{2n+2},
\end{align*}
that is,
\begin{align}\label{eqn01-2}
J_{n}^{\rho,T,1}(l)
&\stackrel{<}{\sim}_{d,n,\rho,T}l^{4n+4} \int_0^{t_{l}} t^{2(2n+2)-(4n+1)}dt \stackrel{<}{\sim}_{d,n,\rho,T} l^{4n}.
\end{align}

On the other hand, from \eqref{eqn02}, we have
\begin{align}
|N^\rho_{n, l,t}|&\leq\left|M^\rho_{n,l,t}\right|+\frac{|c_{n,l}|}{2^n}\left|A^\rho_t\left(\lvert\xi-\cdot\rvert^{2n}\right)(\xi)\right|\cdot |M^\rho_{0,l,t}|\stackrel{<}{\sim}_{d,n,\rho, T}l^{2n}t^{2n}.\label{eqn02-2}
\end{align}
Here we take $\tilde T\in\left(0, T\right)$ and we may  assume that $t_l\leq \tilde T$. Then we split the integration into further two parts:
\begin{align*}
J_{n}^{\rho,T,2}(l)&=\int_{t_l}^{\tilde T} |N^\rho_{n,l,t}|^2\frac{dt}{t^{4n+1}}+\int_{\tilde T}^T |N^\rho_{n,l,t}|^2\frac{dt}{t^{4n+1}} =:J_{n, 1}^{\rho,T,2}(l)+J_{n,2}^{\rho,T,2}(l)
\end{align*}
For $J_{n,2}^{\rho,T,2}(l)$, we get as follows:
\begin{align*}
J_{n,2}^{\rho,T,2}(l)
&=\int_{\tilde T}^T |N^\rho_{n,l,t}|^2\frac{dt}{t^{4n+1}} \stackrel{<}{\sim}_{d,n,\rho,T}  \int_{\tilde T}^T l^{4n}t^{4n}\frac{dt}{t^{4n+1}}\stackrel{<}{\sim}_{d,n,\rho,T, \tilde T} l^{4n}.
\end{align*}

In order to estimate $J_{n,1}^{\rho,T,2}(l)$ part, we improve the estimeat \eqref{m_rough}. We may assume  $\tilde T\leq\frac{\pi}{4}$. Then for $t_l\leq t\leq\tilde T$ and $\tilde T\leq \theta \leq T$, it holds that
$$
\frac{t_l}{T}\tilde T\leq\frac{t}{T}\theta\leq t\leq \tilde T<\frac{\pi}{4}.
$$ 
Therefore we use  \eqref{Psharp} and get
\begin{align*}
&|m^\rho_{l,t}| \stackrel{<}{\sim}_{d,n,\rho,T} t^{2-d}\int_{0}^{T} \left|P_{l,d}\left(\cos\left(\frac{t}{T}\theta\right)\right)\right|\sin ^{d-2}\left(\frac{t}{T}\theta\right)\, \rho\left(\theta\right)d\theta\\
&\stackrel{<}{\sim}_{d,n,\rho,T}t^{2-d}\left\{\int_0^{\frac{\tilde T}{T}t_l}\sin^{d-2}\left(\frac{t}{T}\theta\right)\rho(\theta)d\theta+\int_{\frac{\tilde T}{T}t_l}^T l^{\frac{2-d}{2}}\left(\frac{t}{T}\theta\right)^{\frac{d-2}{2}}\rho(\theta)d\theta\right\}\\
&\stackrel{<}{\sim}_{d,n,\rho,T}\int_0^{\frac{\tilde T}{T}t_l}\theta^{d-2}\rho(\theta)d\theta+l^{\frac{2-d}{2}}\cdot t^{\frac{2-d}{2}}\cdot \int_{\frac{\tilde T}{T}t_l}^T \theta^{\frac{d-2}{2}}\rho(\theta)d\theta\\
&\stackrel{<}{\sim}_{d,n,\rho,T,\tilde T}l^{2-d}+l^{\frac{2-d}{2}}\cdot t^{\frac{2-d}{2}}.
\end{align*}
Therefore,  from \eqref{eqn02}, we get 
\begin{align}
&|N^\rho_{n, l,t}|\leq\left|M^\rho_{n-1,l,t}\right|+\frac{|c_{n,l}|}{2^n}\left|A^\rho_t\left(\lvert\xi-\cdot\rvert^{2n}\right)(\xi)\right|\cdot |m^\rho_{n,l}|\\
&\stackrel{<}{\sim}_{d,n,\rho, T}l^{2(n-1)}t^{2(n-1)}\nonumber\\
&\qquad+l^{2n}\cdot t^{2n}\cdot\left(l^{2-d}+l^{\frac{2-d}{2}}\cdot t^{\frac{2-d}{2}}\right) \nonumber\\
&\stackrel{<}{\sim}_{d,n,\rho, T}l^{2(n-1)}t^{2(n-1)}+l^{2n-d+2}t^{2n}+l^{2n-\frac{d}{2}+1}t^{2n-\frac{d}{2}+1}.\label{eqn02}
\end{align}
Consequently we get
\begin{align}
J^{\rho, T, 2}_{n,1}(l)&\stackrel{<}{\sim}_{d,n,\rho, T}\int_{t_l}^{\tilde T}|N^\rho_{n,l,t}|^2\frac{dt}{t^{4n+1}}\\
&\stackrel{<}{\sim}_{d,n,\rho, T}\int_{t_l}^{\tilde T}\left(l^{4(n-1)}t^{-3}+l^{4n-2d+2}t^{-1}+l^{4n-d+2}t^{-d+1}\right)dt\\
&\stackrel{<}{\sim}_{d,n,\rho, T}l^{4n-2}+l^{4n-2d+3}+l^{4n}\stackrel{<}{\sim}_{d,n,\rho, T}l^{4n}
\end{align}
because $d\geq 2$.
\end{proof}

\subsection{Preparation for lower estimates}
The argument for lower estimate is rather complicated. We start from preparing another estimates of $P_{l,d}(s)$.
\begin{lem}\label{lbLegendre}
For $n\in\N$ and $l\in\N$ such that $l\geq n$, we set
$$
k_{d,n}(l):=
\left\{
\begin{array}{ll}
\frac{1}{2}  &\quad (l=n) \\
\frac{2n+d-1}{(l+n+d-2)(l-n)}\left(\leq 1\right)  &\quad (l\geq n+1)
\end{array}
\right.
$$
Then, for every $\epsilon\in(0,1)$, it holds that 
\begin{equation}\label{Leg08e}
P_{l,d}^{(n)}(s)\geq(1-\vep)P_{l,d}^{(n)}(1)>0\quad\text{for $1-\vep k_{l,d,n}\leq s\leq 1$.}
\end{equation}
\end{lem}
\begin{proof}
The conclusion is obvious for $l=n$ since $P_{l,d}(s)$ will become constant for the $l$-th derivative. Suppose $l\geq n+1$. The mean value theorem and \eqref{eq20} for $k=2$ give us
\begin{equation}\label{Leg08f}
P^{(n)}_{l,d}(s) \geq P^{(n)}_{l,d}(1) - P_{l,d}^{(n+1)}(1) (1-s).
\end{equation}
From \eqref{eq19b}, we have
\[
\frac{P_{l,d}^{(n)}(1)}{P^{(n+1)}_{l,d}(1)}=\frac{2n+d-1}{(l+n+d-2)(l-n)}=k_{d,n}(l) \leq 1
\]
for $l\geq n+1$. Therefore, for
$$
s\geq 1-\vep\frac{P_{l,d}^{(n)}(1)}{P^{(n+1)}_{l,d}(1)}=1-\vep k_{d,n}(l),
$$
it holds that
\begin{align*}
P^{(n)}_{l,d}(s) &\geq P^{(n)}_{l,d}(1) - P_{l,d}^{(n+1)}(1) (1-s)\geq(1-\vep)P_{l.d}^{(n)}(1)>0.
\end{align*}
\end{proof}
As an easy consequence of the definition of $k_{d,n}(l)$, we get as follows: 
\begin{lem}\label{angledef}
For $n\in\N$, $l\in\N$ such that $l\geq n$, and $\vep\in(0,1/2)$, we set $a_{d,n,\vep }(l)\in\left(0,\frac{\pi}{3}\right]$ satisfying
$$
\cos a_{d,n,\vep}(l)= 1-\vep k_{d,n}(l).
$$
Then it holds that
$$
a_{d,n,\vep }(l)\sim \sqrt{\vep k_{d,l}(n)}\sim \vep^{\frac{1}{2}}\left\{l(l+d-2)\right\}^{-\frac{1}{2}}\sim_d \vep^{\frac{1}{2}}l^{-1}.
$$
\end{lem}
We note that we are able to assume that $a_{d,n,\vep}(l)\leq\pi/3$ because  $\vep k_{d,n}(l)\leq \vep\leq\frac{1}{2}$.

Summarizing the above estimate, we get the following fact which we often use:
\begin{lem}\label{est_for_mean}
For $n\in\N$ and $\vep\in(0,1)$, it holds that
\begin{equation}\label{Leg08e}
P_{l,d}^{(n)}(s)\geq(1-\vep)P_{l,d}^{(n)}(1)>0\quad\text{for $s\in(\cos a_{d,n,\vep}(l),1)$.}
\end{equation}
\end{lem}

\subsection{Lower bound of $I_{n,\alpha}^{\rho,T}(l)$}
\begin{lem}\label{LocLowerboundI}
For $n\in\N\cup\{0\}$, $l\in\N$, and $2n < \alpha < 2(n+1)$, it holds that
$$
I_\alpha^{\rho,T}(l)\stackrel{>}{\sim}_{d,\rho,T} l^{2\alpha}.  
$$
\end{lem}
\begin{proof}
By using Lemma \ref{est_for_mean} with a fixed $\vep$, say \(\vep = \frac{1}{2}\). Then for $0\leq \theta\leq a_{d,n+1,\frac{1}{2}}(l)\frac{T}{t}$, it holds that $0\leq\frac{t}{T}\theta\leq a_{d,n+1,\frac{1}{2}}(l)$. Therefore
$$
\cos\left(\frac{t}{T}\theta\right)\geq \cos a_{d,n+1,\frac{1}{2}}(l)=1-\frac{1}{2}k_{d,n+1}(l)
$$
and
$$
P_{l,d}^{(n+1)}(\xi)\geq \frac{1}{2}P_{l.d}^{(n+1)}(1)> 0 \quad\text{for $\xi\in\left(\cos\left(\frac{t}{T}\theta\right),1\right)$}.
$$
Using \eqref{eq20} and \eqref{MTaylor}, we get
\begin{align*}
|M^\rho_{n, l,t}|&\stackrel{>}{\sim}_{d,n,\rho,T} t^{2-d} l^{2(n+1)}\\
&\quad\times\int_0^{\min\left(T,a_{d,n+1,\frac{1}{2}}(l)\frac{T}{t}\right)}\left(1-\cos\left(\frac{t}{T}\theta\right)\right)^{n+1}\sin ^{d-2}\left(\frac{t}{T}\theta\right)  \, \rho(\theta)d\theta\aka{.}
\end{align*}
Moreover, for $t$ satisfying 
\begin{equation}
\label{a-range}
\min(T,a_{d,n+1,\frac{1}{2}}(l)T/t)\geq T_0\quad\Leftrightarrow \quad t\leq \frac{T}{T_0}a_{d,n+1,\frac{1}{2}}(l),
\end{equation}
it holds that
\begin{align}
&\int_0^{\min(T,a_{d,n+1,\frac{1}{2}}(l)T/t)} \left(1-\cos\left(\frac{t}{T}\theta\right) \right)^{n+1}\sin^{d-2}\left(\frac{t}{T}\theta\right) \, \rho(\theta)d\theta\nonumber\\
&\geq \frac{1}{2^{n+1}}\int_{t_0}^{T_0} \sin^{2n+d}\left(\frac{t}{T}\theta\right) \, \rho(\theta)d\theta\stackrel{>}{\sim}_{d,\rho,T} t^{2n+d}\int_{t_0}^{T_0} \theta^{2n+d}\,\rho(\theta)d\theta\nonumber\\
&\stackrel{>}{\sim}_{d,\rho,T}t^{2n+d}.
\end{align}
Here we used the fact that $\sin\theta\geq\frac{\sin T_0}{T_0}\theta$ for $\theta\in [0,T_0]$.
Consequently, we get 
\begin{align*}
M^\rho_{n,l,t}&\stackrel{>}{\sim}_{d,n,T,\rho} l^{2(n+1)} t^{2(n+1)} 
\end{align*}
for $t$ in the range \eqref{a-range}. Therefore it follows that
\begin{align*}
I_\alpha^{\rho,T}(l) &= \int_{0}^{T} \left| M^\rho_{n,l,t} \right|^2 \frac{dt}{t^{2\alpha+1}} \geq \int_{0}^{\frac{T}{T_0}a_{d,n+1,\frac{1}{2}}(l)} \left| M^\rho_{n,l,t} \right|^2 \frac{dt}{t^{2\alpha+1}} \nonumber\\
&\stackrel{>}{\sim}_{d,n,\rho,T} l^{4(n+1)}\int_{0}^{\frac{T}{T_0}a_{d,n+1,\frac{1}{2}}(l)} t^{4(n+1)-2\alpha-1} \, dt\\
&\stackrel{>}{\sim}_{d,n,\rho,T} l^{4(n+1)} a_{d,n+1,\frac{1}{2}}(l)^{4(n+1)-2\alpha}\stackrel{>}{\sim}_{d,n,\rho,T}l^{2\alpha}.
\end{align*}
\end{proof}

\subsection{Lower estimate of $J_{n}^{\rho,T}(l)$}
Before we state the conclusion for this part, we introduce the following condition on $\rho\in W(T,t_0,T_0)$:
\begin{equation}\label{r-cond-fine}
\frac{d-1}{d+2n-1}\cdot\frac{\left(\int_0^{T}\theta^{d-2}\rho(\theta)d\theta\right)^2\int_0^{T}\theta^{2n+d}\rho(\theta)d\theta}{\int_0^{T_0}\theta^{d-2}\rho(\theta)d\theta\int_0^{T_0}\theta^{d}\rho(\theta)d\theta\int_0^{T_0}\theta^{2n+d-2}\rho(\theta)d\theta}<1.
\end{equation}
We will give examples that satisfy this complicated condition \eqref{r-cond-fine}, see Proposition \ref{rho-example}.

\begin{lem}\label{LocLowerboundJ}
Suppose that $\rho\in W(T,t_0,T_0)$ satisfying \eqref{r-cond-fine}. Then for $l, n \in\N$, it holds that
$$
J_{n}^{\rho,T}(l)\stackrel{>}{\sim}_{d,n,\rho,T}l^{4n}.  
$$ 
\end{lem}
In order to prove Lemma \ref{LocLowerboundJ}, we have to sharpen the argument in the proof of Lemma \ref{LocLowerboundI}.
\begin{lem}
For every $\vep\in\left(0,1\right]$ and $l\geq 2$,  we have
\begin{align}
&\cos\theta \geq \cos a_{d,1,\vep}(l)\geq 1-\vep \label{eq25a}\\
&\sin\theta\geq (1-\vep)\theta \label{sinebound}
\end{align}
for every $\theta\in [0,  a_{d,1,\vep}(l)]$. Moreover $a_{d,1,\vep}(l)\in \left(0,\frac{\pi}{3}\right]$ if $\vep\in\left(0,\frac{1}{2}\right]$.
\end{lem}
\begin{proof}
$a_{d,1,\vep}(l)$ for $l \geq 2$ satisfies
\begin{align}
\cos a_{d,1,\vep}(l)&=1-\vep\frac{d+1}{(l+d-1)(l-1)}\geq 1-\vep.
\end{align}
Therefore 
$$
\sin\theta=\int_0^\theta\cos \theta d\theta\geq (1-\vep)\theta.
$$
\end{proof}
Using this lemma, it is obvious that the following holds:
\begin{lem}
\label{sharpz}
For every $\rho\in W(T,t_0,T_0)$, it holds that
\begin{equation}\label{z-behavior-fine}
 (1-\vep)^{d-2}c_z t^{d-1}\leq z_t\leq C_z t^{d-1}\quad\text{ for every $t\in\left[0,\frac{T}{T_0}a_{d,1,\vep}(l)\right]$},
\end{equation}
where
\begin{align*}
c_z&=\left|\sphere{d-2}\right|\cdot T^{1-d}\int_{0}^{T_0}\theta^{d-2}\rho(\theta)d\theta,\\
C_z&=\left|\sphere{d-2}\right|\cdot T^{1-d}\int_0^T\theta^{d-2}\rho(\theta)d\theta.
\end{align*}
\end{lem}

\begin{proof}[Proof of Lemma {\ref{LocLowerboundJ}}]
From Lemma \ref{est_for_mean} for some $\vep\in\left(0,\frac{1}{2}\right]$, which we determine later, it holds that
\begin{equation}\label{eq24a}
P_{l,d}^{(1)}(s)\geq \left(1-\vep\right)	P_{l,d}^{(1)}(1), \hspace{0,5cm}
\quad\text{for $s\in(\cos a_{d,1,\vep}(l), 1)$.}
\end{equation}
By using \eqref{N-decomp}, we get
\begin{align}
J^{\rho, T}_n(l)^\frac{1}{2}&\geq\left\{\int_0^{a_{d,1,\vep}(l)}\left|N_{l,t}^\rho\right|^2\frac{dt}{t^{4n+1}}\right\}^\frac{1}{2}
\nonumber\\
&\geq\left\{\int_0^{a_{d,1,\vep}(l)} \left|
\frac{c_{n,l}}{2^n}A^\rho_t\left(\lvert\xi-\cdot\rvert^{2n}\right)(\xi)M^\rho_{0,l,t}\right|^2\frac{dt}{t^{4n+1}}\right\}^\frac{1}{2}\nonumber\\
&\qquad-\left\{\int_0^{a_{d,1,\vep}(l)} \left|M_{n,l,t}^{\rho}\right|^2\frac{dt}{t^{4n+1}}\right\}^\frac{1}{2},\label{eq23a}
\end{align}

Similar to the proof of Lemma \ref{LocLowerboundI}, when $0\leq t\leq a_{d,1,\vep}(l)$, it holds that 
$$
0\leq \frac{t}{T}\theta\leq \frac{ a_{d,1,\vep}(l)}{T}\theta\leq \frac{T_0}{T} a_{d,1,\vep}(l)\leq a_{d,1,\vep}(l)
$$
when $0\leq\theta \leq T_0$. Therefore, using \eqref{sinebound}, we are able to refine the estimate of $M_{0,l,t}^\rho$ as follows:
\begin{align*}
|M^\rho_{0, l,t}|
&\geq \frac{\left|\sphere{d-2}\right|}{z_t}\cdot\frac{t}{T}\cdot(1-\vep)P^{(1)}_{l,d}(1)\int_{0}^{T_0}\left(1-\cos\left(\frac{t}{T}\theta\right)\right) \\
&\qquad\times\sin ^{d-2}\left(\frac{t}{T}\theta\right)  \, \rho(\theta)d\theta\\
&\geq \frac{\left|\sphere{d-2}\right|}{z_t}\cdot\frac{t}{T}\cdot(1-\vep)P^{(1)}_{l,d}(1)\int_{0}^{T_0}\frac{\sin ^d\left(\frac{t}{T}\theta\right)}{1+\cos \left(\frac{t}{T}\theta\right)}\, \rho(\theta)d\theta\\
&\geq(1-\vep)^{d+1}\cdot\frac{\left|\sphere{d-2}\right|}{z_t}\cdot \frac{P^{(1)}_{l,d}(1)}{2}\cdot \left(\frac{t}{T}\right)^{d+1}\int_0^{T_0}\theta^d\rho(\theta)d\theta.
\end{align*}
Then we have 
\begin{align*}
&\left|\frac{c_{n,l}}{2^n}A^\rho_t\left(\lvert\xi-\cdot\rvert^{2n}\right)(\xi)M_{0,l,t}^\rho\right|\\
&\geq|c_{n,l}|\cdot\frac{ |\sphere{d-2}|}{z_t}\cdot\frac{t}{T}\cdot|M^\rho_{0,l,t}|\int_{0}^{T_0}\frac{\sin^{2n+d-2}\left(\frac{t}{T}\theta\right)}{\left(1+\cos\left(\frac{t}{T}\theta\right)\right)^n}\rho(\theta)d\theta\nonumber\\
&\geq(1-\vep)^{2n+2d-1}\left(\frac{\left|\sphere{d-2}\right|}{z_t}\right)^2\cdot\frac{P^{(1)}_{l,d}(1)P^{(n)}_{l,d}(1)}{2^{n+1}n!}\left(\frac{t}{T}\right)^{2n+2d}\nonumber\\
&\qquad\times\int_0^{T_0}\theta^d\rho(\theta)d\theta\int_0^{T_0}\theta^{2n+d-2}\rho(\theta)d\theta\\
&\geq(1-\vep)^{2n+2d-1}\left(\frac{\left|\sphere{d-2}\right|}{C_z}\right)^2\cdot\frac{P^{(1)}_{l,d}(1)P^{(n)}_{l,d}(1)}{2^{n+1}n!}T^{-2n-2} t^{2n+2}\nonumber\\
&\qquad\times\int_0^{T_0}\theta^d\rho(\theta)d\theta\int_0^{T_0}\theta^{2n+d-2}\rho(\theta)d\theta\\
&=:(1-\vep)^{2n+2d-1}C_Nt^{2n+2},
\end{align*}
where
$$
C_N=\left(\frac{\left|\sphere{d-2}\right|}{C_z}\right)^2\cdot\frac{P^{(1)}_{l,d}(1)P^{(n)}_{l,d}(1)}{2^{n+1}n!}T^{-2n-2}\int_0^{T_0}\theta^d\rho(\theta)d\theta\int_0^{T_0}\theta^{2n+d-2}\rho(\theta)d\theta.
$$
Then by using \eqref{eq24a} and \eqref{eq25a}
\begin{align*}
&\int_0^{a_{d,1,\vep}(l)}\left|\frac{c_{n,l}}{2^n}A^\rho_t\left(\lvert\xi-\cdot\rvert^{2n}\right)(\xi)M_{0,l,t}^\rho\right|^2\frac{dt}{t^{4n+1}}\\
&\geq (1-\vep)^{4d+4n-2}C_N^2\int_0^{a_{d,1,\vep}(l)}t^{3}dt\\
&= (1-\vep)^{4d+4n-2}\cdot\frac{C_N^2}{4}\cdot a_{d,1,\vep}(l)^4.
\end{align*}

Next, we estimate from above the second integral of \eqref{eq23a}. Using \eqref{eq20} and \eqref{NTaylor3}, we have
\begin{align*}
\left|M_{n,l,t}^{\rho}\right|&\leq \frac{ |\sphere{d-2}|}{z_t}\cdot\frac{t}{T}\cdot\frac{P^{(n+1)}_{l,d}(1)}{n!}\\
&\qquad\times\int_0^T\left(1-\cos\left(\frac{t}{T}\theta\right)\right)^{n+1}\sin ^{d-2}\left(\frac{t}{T}\theta\right)  \, \rho(\theta)d\theta\\
&\leq\frac{\left|\sphere{d-2}\right|}{z_t}\cdot\frac{t}{T}\cdot\frac{P^{(n+1)}_{l,d}(1)}{n!}\int_0^T\frac{\sin ^{2n+d}\left(\frac{t}{T}\theta\right)}{\left(2\cos \left(\frac{t}{T}\theta\right)\right)^{n+1}}  \, \rho(\theta)d\theta\\
&\leq(1-\vep)^{2-d-(n+1)}\frac{\left|\sphere{d-2}\right|}{c_z t^{d-1}}\cdot\frac{P^{(n+1)}_{l,d}(1)}{n!2^{n+1}}\cdot \left(\frac{t}{T}\right)^{2n+d+1}\int_0^T\theta ^{2n+d}\, \rho(\theta)d\theta\\
&=:(1-\vep)^{1-d-n}D_Nt^{2n+2}
\end{align*}
since $0\leq\frac{t}{T}\theta\leq a_{d,1,\vep}(l)$, where
$$
D_N=\frac{\left|\sphere{d-2}\right|}{c_z}\cdot\frac{P^{(n+1)}_{l,d}(1)}{n!2^{n+1}}\cdot T^{-2n-d-1}\int_0^T\theta ^{2n+d} \rho(\theta)d\theta
$$
Then
\begin{align*}
\int_0^{a_{d,1,\vep}(l)} \left|M_{n,l,t}^{\rho}\right|^2\frac{dt}{t^{4n+1}}&\leq(1-\vep)^{2-2d-2n} D_N^2\int^{a_{d,1,\vep}(l)}_0t^3dt\\
&=(1-\vep)^{2-2d-2n}\cdot D_N^2\cdot\frac{a_{d,1,\vep}(l)^4}{4}.
\end{align*}

By using \eqref{eq19b}, we get
\begin{align*}
&J^{\rho,T}_n(l)^{\frac{1}{2}}\\
& \geq (1-\vep)^{2d+2n-1}\cdot\frac{C_N}{2}\cdot a_{d,1,\vep}(l)^2-(1-\vep)^{1-d-n}\cdot\frac{D_N}{2}\cdot a_{d,1,\vep}(l)^2\nonumber\\
&=(1-\vep)^{1-d-n}\cdot\frac{C_N}{2}\cdot a_{d,1,\vep}(l)^2\left\{(1-\vep)^{3d+3n-2}  -\frac{D_N}{C_N}   \right\}.
\end{align*}
Here
\begin{align*}
\frac{D_N}{C_N}
&=\frac{P^{(n+1)}_{l,d}(1)}{P^{(1)}_{l,d}(1)P^{(n)}_{l,d}(1)}\cdot \frac{\left(\int_0^{T}\theta^{d-2}\rho(\theta)d\theta\right)^2\int_0^{T}\theta^{2n+d}\rho(\theta)d\theta}{\int_0^{T_0}\theta^{d-2}\rho(\theta)d\theta\int_0^{T_0}\theta^{d}\rho(\theta)d\theta\int_0^{T_0}\theta^{2n+d-2}\rho(\theta)d\theta}
\end{align*}
and
$$
\frac{P^{(n+1)}_{l,d}(1)}{P^{(1)}_{l,d}(1)P^{(n)}_{l,d}(1)}=\frac{(d-1)(d+l+n-2)(l-n)}{(d+2n-1)(d+l-2)l}\tend\frac{d-1}{d+2n-1}<1.
$$
Therefore, 
$$
\lim_{l\tend\infty}\frac{D_N}{C_N}<1
$$
if \eqref{r-cond-fine} holds.
By taking \(\vep\) dependently only on \(n\) and \(d\) such that
\begin{equation*}
(1-\vep)^{3d+3n-2}  >\lim_{l\tend\infty}\frac{D_N}{C_N}.
\end{equation*}
Then the proof of \(J^{\rho,T}_n(l)\stackrel{>}{\sim}l^{4n}\) is complete since $C_N\sim P_{l,d}^{(n)}(1)\sim l^{2n}$.
\end{proof}

Here we give examples that satisfy \eqref{r-cond-fine}.
\begin{prop}\label{rho-example}
Suppose $0<T\leq\pi$ and $0<t_0<T_0<\pi$. Then the indicator functions $I_{[0.T]}$ and $I_{[t_0,T_0]}$ for the intervals $[0,T]$ and $[t_0,T_0]$, respectively, belong to $W(T,t_0,T_0)$ and satisfy \eqref{r-cond-fine}. 
\end{prop}
\begin{proof}
We may only check $\rho=I_{[t_0,T_0]}$ case. In this case, 
\begin{align*}
&\frac{d-1}{d+2n-1}\cdot\frac{\left(\int_0^{T}\theta^{d-2}\rho(\theta)d\theta\right)^2\int_0^{T}\theta^{2n+d}\rho(\theta)d\theta}{\int_0^{T_0}\theta^{d-2}\rho(\theta)d\theta\int_0^{T_0}\theta^{d}\rho(\theta)d\theta\int_0^{T_0}\theta^{2n+d-2}\rho(\theta)d\theta} \\
	&=\frac{d-1}{d+2n-1}\cdot \frac{\left\{\frac{1}{d-1}\left(T_\rho^{d-1}-t_\rho^{d-1}\right)\right\}^2\frac{1}{2n+d+1}\left(T_\rho^{2n+d+1}-t_\rho^{2n+d+1}\right)}{\frac{1}{d-1}\left(T_\rho^{d-1}-t_\rho^{d-1}\right)\frac{1}{d+1}\left(T_\rho^{d+1}-t_\rho^{d+1}\right)\frac{1}{2n+d-1}\left(T_\rho^{2n+d-1}-t_\rho^{2n+d-1}\right)}\\
	&=\frac{d+1}{2n+d+1}\cdot\frac{\left(T_\rho^{d-1}-t_\rho^{d-1}\right)\left(T_\rho^{2n+d+1}-t_\rho^{2n+d+1}\right)}{\left(T_\rho^{d+1}-t_\rho^{d+1}\right)\left(T_\rho^{2n+d-1}-t_\rho^{2n+d-1}\right)}\\
	&=\frac{d+1}{2n+d+1}\cdot\frac{T_\rho^{2n+2d}+t_\rho^{2n+2d}-T_\rho^{d-1}t_\rho^{2n+d+1}-T_\rho^{2n+d+1}t_\rho^{d-1}}{T_\rho^{2n+2d}+t_\rho^{2n+2d}-T_\rho^{d+1}t_\rho^{2n+d-1}-T_\rho^{2n+d-1}t_\rho^{d+1}}\\
	&=\frac{d+1}{2n+d+1}\cdot \frac{1+r^{2n+2d}-r^{2n+d+1}-r^{d-1}}{1+r^{2n+2d}-r^{2n+d-1}-r^{d+1}},
\end{align*}
where we set $r=\frac{t_\rho}{T_\rho}\in(0,1)$.

It is easy to check that 
\begin{equation*}
\frac{r^{2n+d+1}+r^{d-1}}{r^{2n+d-1}+r^{d+1}}=\frac{r^{n+1}+\frac{1}{r^{n+1}}}{r^{n-1}+\frac{1}{r^{n-1}}}>1,
\end{equation*}
by featuring the function \(x+\frac{1}{x}\) and a fact that \(r^{n+1}<r^{n-1}\).
So we have
\begin{equation*}
	\frac{1+r^{2n+2d}-r^{2n+d+1}-r^{d-1}}{1+r^{2n+2d}-r^{2n+d-1}-r^{d+1}}<1.
\end{equation*}
Since \(\frac{d+1}{2n+d+1}<1\), then the proof is complete. 
\end{proof}

\subsection{The final part of the proof of Main Theorem}
Most of the part of the proof of the main theorem (Theorem \ref{main}) has already been completed. In fact, the part from (1) to (2) follows  from Proposition \ref{equinorm}, since we may take $g_k=T_k((-\lap)^k f)$ if $f\in H^\alpha(\sphere{d-1})$. On the other hand, the part from (2) to (1) follows from an almost identical argument presented in \cite{blp20}. Nevertheless, we include it here for the sake of completeness.

We start from recalling several facts about approximating $f\in L^2(\mathbb{S}^{d-1})$ by smooth functions via the Poisson transform:
\[
P_r f(\xi):=\int_{\mathbb{S}^{d-1}}p_r(\eta,\xi)f(\eta)d\sigma(\eta),\quad
p_r(\eta,\xi):=\frac{1-r^2}{|\mathbb{S}^{d-1}||r\xi-\eta|}
\]
for $0<r<1$. 

The Poisson kernel is known to be decomposed as follows, see \cite[5.28 (p.95) and Theorem 5.33 ]{abr01} or \cite[Theorem 2.9 and Proposition 2.28]{ah12}:
\begin{prop}[Poisson identity]\label{Poisson}
\begin{equation}\label{PoissonID}
p_r(\eta,\xi)=\sum_{l=0}^\infty r^l\sum_{j=1}^{\nu(l)}\clo{Y_l^j(\eta)}Y_l^j(\xi),
\end{equation}
that means
$$
P_r f(\xi)=\sum_{l=0}^\infty r^l\sum_{j=1}^{\nu(l)}\hat{f}_{lj}Y_l^j(\xi)
$$
when $f=\sum_{j=1}^{\nu(l)}\hat{f}_{lj}Y_l^j$ for every $f\in \Le{2}{\sphere{d-1}}$.
\end{prop}
As a consequence of above, the following holds:
\begin{prop}[{\cite[Theorem 6.4 and 6.7]{abr01}}]\label{Prfrelation}
For $f\in L^2(\mathbb{S}^{d-1})$, the following hold:
\begin{enumerate}
\item
$\|P_r f\|_{L^2(\mathbb{S}^{d-1})}\leq \|f\|_{L^2(\mathbb{S}^{d-1})}$. 
\item
$\|P_rf-f\|_{L^2(\mathbb{S}^{d-1})}\to 0$ as $r\uparrow 1$.
\end{enumerate}
\end{prop}
More generally, the following holds:
\begin{prop}[cf. {\cite[iii) in p.9]{blp20}}]
$P_r f\in C^\infty(\sphere{d-1})(\subset H^\alpha(\sphere{d-1}))$ and it holds that
$$
\|P_r f\|_{H^\alpha (\mathbb{S}^{d-1})}\leq \|f\|_{H^\alpha(\mathbb{S}^{d-1})}
$$ for any $\alpha\geq 0$.
\end{prop}

The following result concerning the relation between $P_rf$ and  $A_t^\rho f$ follows directly from the Poisson identity \eqref{PoissonID}:
\begin{prop}[cf. {\cite[p.29]{blp20}}] \label{commute}
For $f\in\Le{2}{\mathbb{S}^{d-1}}$, it holds that
\[
A^\rho_t (P_r f)=P_r(A^\rho_t f)=\sum_{l=0}^\infty r^lm^\rho_{l,t}\sum_{j=1}^{\nu(l)}\hat{f}_{lj}Y_l^j(\xi)\in C^\infty(\sphere{d-1}).
\]
In particlular we have
\[
A^\rho_t f=\lim_{t\uparrow 1}A^\rho_t(P_r f)\quad\text{in $L^2(\mathbb{S}^{d-1})$}.
\]
\end{prop}
The following result then follows immediately from the Minkowski inequality: 
\begin{prop}[cf. {\cite[(29)]{blp20}}] \label{P-square}
For $f, g_1,\cdots, g_n\in\Le{2}{\mathbb{S}^{d-1}}$ satisfying
\begin{equation}
\label{S-L2}
S^{\rho,T}_\alpha (f,g_1,\cdots, g_n)(\xi)\in\Le{2}{\sphere{d-1}},
\end{equation}
the both sides of the following inequality are well-defined and satisfies
\[
S^{\rho,T}_\alpha(P_r f, P_r g_1,\cdots, P_r g_n)(\xi)\leq P_rS^{\rho,T}_\alpha (f,g_1,\cdots, g_n)(\xi)\;\text{for all $\xi\in\sphere{d-1}$.}
\]
\end{prop}
This result leads to the following conclusion,  which constitutes the main part of the proof of the implication from (2) to (1), by an argument similar to that in \cite{blp20}:
\begin{prop}[cf. {\cite[(30)]{blp20}}] \label{g-regularize}
For $f, g_1,\cdots, g_n\in\Le{2}{\mathbb{S}^{d-1}}$ satisfying \eqref{S-L2}, it holds that
\begin{equation}\label{Pg}
P_rg_k(\xi)=T_k((-\lap)^k P_rf)(\xi)
\end{equation}
for all $r\in(0,1)$, $\xi\in\sphere{d-1}$, and $k=1,\cdots,n$.
\end{prop}
\begin{proof}
We set
\begin{equation}\label{H-diverge}
H_k(\xi):=P_rg_k(\xi)-T_k((-\lap)^k P_rf)(\xi)
\end{equation}
and
$$
K(t,\xi):=\sum_{k=1}^nH_k(\xi)A^\rho_t(|\xi-\cdot|^{2k})(\xi),
$$

We have already done for $0<\alpha<2$ case in \cite{s18}, see Theorem \ref{themainshortpaper}.

We start from the case $\alpha\not=2n\geq 2$. Then it holds that
\begin{align*}
&\left(\int_0^TK(t,\xi)^2\frac{dt}{t^{2\alpha+1}}\right)^{\frac{1}{2}}
\leq S^{\rho,T}_\alpha(P_r f, P_r g_1,\cdots, P_r g_n)(\xi)\\
&\qquad\qquad +S^{\rho,T}_\alpha (P_r f,T_1((-\lap) P_rf),\cdots, T_n((-\lap)^n P_rf))(\xi)\\
&\leq P_r  S^{\rho,T}_\alpha(f, g_1,\cdots, g_n)(\xi)\\
&\qquad\qquad+S^{\rho,T}_\alpha (P_r f,T_1((-\lap) P_rf),\cdots, T_n((-\lap)^n P_rf))(\xi)<\infty
\end{align*}
from Proposition \ref{P-square} and $P_rf\in C^\infty(\sphere{d-1})$. 

On the other hand, \eqref{rough-A} guarantees that
$$
K(t,\xi)\sim_{d,n,\rho,T}H_1(\xi)t^2+\cdots+H_n(\xi)t^{2n}.
$$
Therefore we may assume that
$$
K(t,\xi)\stackrel{>}{\sim}_{d,\rho,T}H_1(\xi)t^2.
$$
Then we get
\begin{align*}
\left|H_1(\xi)\right|\left(\int_0^T\frac{dt}{t^{2\alpha-3}}\right)^{\frac{1}{2}}&=\left\{\int_0^T\left|H_1(\xi)t^2\right|^2\frac{dt}{t^{2\alpha+1}}\right\}^{\frac{1}{2}}\\
&\stackrel{<}{\sim}_{d,k,\rho,T}\left(\int_0^TK(t,\xi)^2\frac{dt}{t^{2\alpha+1}}\right)^{\frac{1}{2}}<\infty.
\end{align*}
This implies $H_1(\xi)=0$ because $\int_0^T\frac{dt}{t^{2\alpha-3}}=\infty$ when $\alpha\geq 2$. Then, by induction, we obtain
$$
H_1(\xi)=\cdots=H_n(\xi)=0,
$$
which completes the proof.

The proof for the case $\alpha=2n\geq 2$ proceeds in a similar manner. We introduce
$$
\tilde H(\xi):=A^\rho_t(P_r g_n)(\xi)-A^\rho_t(T_n((-\lap)^n P_rf))(\xi)
$$
and
$$
L(t,\xi):=\sum_{k=1}^{n-1}H_k(\xi)A^\rho_t(|\xi-\cdot|^{2k})(\xi)+\tilde H(\xi)A^\rho_t(|\xi-\cdot|^{2n})(\xi).
$$
Then we get
$$
\left(\int_0^TL(t,\xi)^2\frac{dt}{t^{2\alpha+1}}\right)^{\frac{1}{2}}<\infty
$$
from the assumption. On the other hand, it holds that
$$
L(t,\xi)\sim_{d,n,\rho,T}H_1(\xi)t^2+\cdots+H_{n-1}(\xi)t^{2n-2}+\tilde H(\xi) t^{2n}.
$$
Since $n\geq 2$, we get
$$
H_1(\xi)=\cdots=H_{n-1}(\xi)=0,
$$
which implies \eqref{Pg} for $k=1,\cdots,n-1$. Moreover, we obtain
$$
\tilde H(\xi)=0\quad\Leftrightarrow\quad A^\rho_t(P_r g_n)(\xi)=A^\rho_t(T_n((-\lap)^n P_rf))(\xi).
$$
Since $P_r g_n$ and $T_n((-\lap)^n P_rf\in C^\infty(\sphere{d-1})$, it holds that
$$
\lim_{t\downarrow 0} A^\rho_t(P_r g_n)(\xi)=P_r g_n(\xi)
$$
and
$$
\lim_{t\downarrow 0}  A^\rho_t(T_n((-\lap)^n P_rf))(\xi)=T_n((-\lap)^n P_rf)(\xi).
$$
Thus, we obtain the result for the case $\alpha = 2n$.
\end{proof}

\begin{proof}[Proof of the part from (2) to (1) of Theorem \ref{main}]
From the assumption in (2), Proposition \ref{g-regularize} guarantees \eqref{Pg}. Then we can conclude that the existence of functions $F_k(\xi)\in\Le{2}{\sphere{d-1}}$ such that
$$
g_k(\xi)=T_k (F_k)(\xi)\quad{and}\quad (-\Delta)^{k} P_{r}f\tend F_k(\xi)\quad\text{in $\Le{2}{\sphere{d-1}}$}
$$
as $r\uparrow 1$ by \eqref{Pg} and Proposition \ref{Prfrelation}(2). This follows from the fact that \(T_{k}\) is invertible in \(L^{2}(\mathbb{S}^{d-1})\).  

Consequently, by a standard argument of the distribution theory, we obtain
\[
 (-\Delta)^{k} f = F_{k}\quad\text{in $\calD'(\sphere{d-1})$}
\]
for all $k=1,\cdots,n$, which implies $f\in H^{2n}(\sphere{d-1})$ and 
$$
g_k(\xi)=T_k((-\Delta)^{k} f((\xi)
$$
for all $k=1,\cdots,n$. Thus the conclusion follows from Proposition \ref{equinorm}.
\end{proof}

\section*{Acknowledgement}
This work was supported by Japan Society for the Promotion of Science Grant-in-Aid for Scientific Research Numbers 24K06794.

%
%
%

\begin{thebibliography}{99}

\bibitem{amv12}Roc Alabern, Joan Mateu, and Joan Verdera, A new characterization of Sobolev spaces on $\R^n$, Math. Ann. {\bf 354} (2012) no.2, 589–626.

\bibitem{ah12}
{Kendall Atkinson and Weimin Han, Spherical harmonics and approximations on the unit sphere: an introduction, Lecture Notes in Math., 2044, 2012, Springer, Heidelberg, 2012, x+244 pp.}

\bibitem{abr01}
{Sheldon Axler, Paul Bourdon, and Wade Ramey, Harmonic function theory, Second edition, Grad. Texts in Math., 137, 2044, 2001, Springer, New York, 2001, xii+259 pp.}


\bibitem{blp20}Juan Antonio Barcel\'o, Teresa Luque, and Salvador P\'erez-Esteva, Characterization of Sobolev spaces on the sphere, J. Math. Anal. Appl. {\bf 491} (2020) no.1, 124240, 23 pp.



\bibitem{hl16}Piotr Haj{\l}asz, Sobolev spaces on an arbitrary metric space, Potential Anal. {\bf 5} (1996), no. 4, 403–415.

\bibitem{hl17}Piotr Haj{\l}asz and Zhuomin Liu, A Marcinkiewicz integral type characterization of the Sobolev space, Publ. Mat. {\bf 61} (2017) no.1, 83–104.



\bibitem{s18}Ikhsan Maulidi and Hiroshi Ohtsuka, A generalized averaging of a function and characterization of Sobolev spaces on the sphere, The Science Reports of Kanazawa University  {\bf 68} (2025), 1–9.

\bibitem{s17}Shuichi Sato, Littlewood-Paley equivalence and homogeneous Fourier multipliers, Integral Equations Operator Theory {\bf 87} (2017) no.1, 15–44.
\end{thebibliography}
\end{document}